\documentclass[twoside, 11pt]{article}
\usepackage{mathrsfs}
\usepackage{amssymb, amsmath, mathrsfs, amsthm}%, wasysym}
\usepackage{graphicx}
\usepackage{color}
\usepackage[top=2cm, bottom=2cm, left=2cm, right=2cm]{geometry}
\usepackage{float, caption, subcaption}

\input{epsf}

\newcommand{\bd}{\begin{description}}
\newcommand{\ed}{\end{description}}
\newcommand{\bi}{\begin{itemize}}
\newcommand{\ei}{\end{itemize}}
\newcommand{\be}{\begin{enumerate}}
\newcommand{\ee}{\end{enumerate}}
\newcommand{\beq}{\begin{equation}}
\newcommand{\eeq}{\end{equation}}
\newcommand{\beqs}{\begin{eqnarray*}}
\newcommand{\eeqs}{\end{eqnarray*}}

\definecolor{DarkGreen}{rgb}{0.2, 0.6, 0.3}

\newtheorem{theorem}{Theorem}[section]

\newtheorem{lemma}{Lemma}[section]
\newtheorem{definition}{Definition}
\newtheorem{corollary}[theorem]{Corollary}

\newtheorem{claim}{Claim}
\newtheorem{remark}{Remark}[section]

\begin{document}
\title{\textbf{Gallai-Ramsey numbers involving a rainbow $4$-path} \footnote{Supported by the National
Science Foundation of China (Nos. 12061059, 11601254, 11551001) and
the Qinghai Key Laboratory of Internet of Things Project
(2017-ZJ-Y21).} }

\author{Jinyu Zou
\footnote{Department of Basic Research, Qinghai University, Xining,
810016, China. {\tt jinyuzou@126.com}}, \ \ Zhao
Wang\footnote{College of Science, China Jiliang University, Hangzhou
310018, China. {\tt wangzhao@mail.bnu.edu.cn}}, \ \ Hong-Jian
Lai\footnote{Department of Mathematics, West Virginia University,
Morgantown, WV 26506-6310, USA {\tt hjlai2015@hotmail.com}}, \ \
Yaping Mao\footnote{School of Mathematics and Statistis, Qinghai
Normal University, Xining, Qinghai 810008, China. {\tt
maoyaping@ymail.com}} \ \footnote{Academy of Plateau Science and
Sustainability, Xining, Qinghai 810008, China.} \
\footnote{Corresponding author}}
\date{}
\maketitle

\begin{abstract}
Given two non-empty graphs $G,H$ and a positive integer $k$, the
Gallai-Ramsey number $\operatorname{gr}_k(G:H)$ is defined as the
minimum integer $N$ such that for all $n\geq N$, every
$k$-edge-coloring of $K_n$ contains either a rainbow colored copy of
$G$ or a monochromatic copy of $H$. In this paper, we got some exact
values or bounds for $\operatorname{gr}_k(P_5:H) \ (k\geq 3)$ if $H$ is a general graph or a star with extra independent edges or a pineapple.\\[2mm]
\end{abstract}

\section{Introduction}

All graphs considered are finite, simple and undirected. We follow
the notation and terminology of Bondy \cite{Bondy}. Let $V(G)$,
$E(G)$, $e(G)$, $\delta(G)$ be the vertex set, edge set, size,
minimum degree of graph $G$, respectively. We use $G-X$ to denote
the subgraph of $G$ obtained by removing all the vertices of $X$
together with the edges incident with them from $G$; similarly, we
use $G\setminus M$ to denote the subgraph of $G$ obtained by
removing all the edges of $M$ from $G$. The {\it union} $G\cup H$ of
two graphs $G$ and $H$ is the graph with vertex set $V(G)\cup V(H)$
and edge set $E(G)\cup E(H)$. We call the path $P_{t+1}$, of order
$t+1$ and having $t$ edges, a \emph{$t$-path}. A {\it complete
graph} is a graph in which every pair of vertices are adjacent, and
a complete graph on $n$ vertices is denoted by $K_n$. Let $[a,b]$ be
the interval from $a$ to $b$.

An $k$-edge-coloring is \emph{exact} if all colors are used at least
once. In this work, we consider only exact edge-colorings of graphs.
An edge colored graph is called {\it rainbow} if all edges have
different colors and {\it monochromatic} if all edges have a single
color.

\subsection{Classical Ramsey number}

Ramsey theory were introduced in 1930; see \cite{Ramsey}. The main
subject of the theory are complete graphs whose subgraphs can have
some regular properties.
\begin{definition}
Given $k$ graphs $H_1,H_2,\ldots,H_k$, the \emph{Ramsey number}
$\operatorname{R}(H_1,H_2,\ldots,H_k)$ is defined as the minimum
number of vertices $n$ needed so that every $k$-edge-coloring of
$K_{n}$ contains a monochromatic $H_i$, where $1\leq i\leq n$.
\end{definition}
If $H_1=H_2=\cdots=H_k$, then we write the number as
$\operatorname{R}_k(H)$. If $H_i=K_{k_i}\ (1\leq i\leq k)$, then we
use the abbreviation $\operatorname{R}(k_1,\dots,k_k)$.

Ramsey number has its applications on the fields of communications,
information retrieval in computer science, and decision-making; see
\cite{Roberts, Rosta} for examples. We refer the interested reader
to \cite{MR1670625} for a dynamic survey of small Ramsey numbers.

\subsection{Gallai-Ramsey number}

Edge colorings of complete graphs that contain no rainbow triangle
have very interesting and somewhat surprising structure. In 1967,
Gallai \cite{Gallai1967} first examined this structure under the
guise of transitive orientations of graphs and it can also be traced
back to \cite{CameronEdmonds}. For this reason, colored complete
graphs containing no rainbow triangle are called \emph{Gallai
colorings}. Gallai's result was restated in \cite{GyarfaSimonyi} in
the terminology of graphs. For the following statement, a trivial
partition is a partition into only one part.

\begin{theorem}{\upshape \cite{CameronEdmonds,Gallai1967,GyarfaSimonyi}}\label{Thm:G-Part}
In any coloring of a complete graph containing no rainbow triangle,
there exists a nontrivial partition of the vertices (called a Gallai
partition), say $H_1,H_2,\ldots,H_t$, satisfying the following two
conditions.

$(a)$ The number of colors on the edges among $H_1,H_2,\ldots,H_t$
are at most two.

$(b)$ For each part pair $H_i,H_j \ (1\leq i\neq j\leq t)$, all the
edges between $H_i$ and $H_j$ receive the same color.
\end{theorem}

The induced subgraph of a Gallai colored complete graph constructed
by selecting a single vertex from each part of a Gallai partition is
called the \emph{reduced graph}. By Theorem~\ref{Thm:G-Part}, the
reduced graph is a $2$-colored complete graph. This kind of
restriction on the distribution of colors has led to a variety of
interesting works like \cite{GyarfaPalvolgyi}.

\begin{definition}
Given two graphs $G$ and $H$, the \emph{general $k$-colored
Gallai-Ramsey number} $\operatorname{gr}_k(G:H)$ is defined to be
the minimum integer $m$ such that every $k$-coloring of the complete
graph on $m$ vertices contains either a rainbow copy of $G$ or a
monochromatic copy of $H$.
\end{definition}

With the additional restriction of forbidding the rainbow copy of
$G$, it is clear that $\operatorname{gr}_k(G:H)\leq
\operatorname{R}_k(H)$ for any $G$. Till now, most work focuses on
the case $G=K_3$; see \cite{ChenLiPei, FoxSudakov2008,
FujitaMagnant, GyarfaSimonyi, GyarfaSimonyi, LiWang,
LiBesseMagnantWangWatts, MWMS, MWMS2020, MWMS-Wheels,
WangMaoZouMaganant, ZhaoWei}. For more details on the Gallai-Ramsey
numbers, we refer to the book \cite{MagnantNowbandegani} and a
survey paper \cite{FMO14}.

\subsection{Structural theorems and main results}

Thomason and Wagner \cite{ThomasonWagner2007} obtained the
structural theorems for $P_4$ and $P_5$.

\begin{theorem}{\upshape \cite{ThomasonWagner2007}}\label{th-ThomasonWagner2007-Structure1}
Let $K_n$, $n\geq 4$, be edge colored so that it contains no rainbow
$3$-path $P_4$. Then one of the following holds:

$(a)$ at most two colors are used;

$(b)$ $n=4$ and three colors are used, each color forming a
$1$-factor.
\end{theorem}

In an edge-colored graph, define $V^{(j)}$ as the set of vertices
with at least one incident edge in color $j$ and denote $E^{(j)}$ to
be the set of edges of color $j$ for a given color $j$.

\begin{theorem}{\upshape \cite{ThomasonWagner2007}}\label{th-path-Structure2}
Let $K_n$, $n\geq 5$, be edge colored so that it contains no rainbow
$4$-path $P_5$. Then, after renumbering the colors, one of the
following must hold:

$(a)$ at most three colors are used;

$(b)$ color $1$ is dominant, meaning that the sets $V^{(j)}$, $j\geq 2$,
are disjoint;

$(c)$ $K_n-a$ is monochromatic for some vertex $a$;

$(d)$ there are three vertices $a,b,c$ such that $E^{(2)}=\{ab\}$,
$E^{(3)}=\{ac\}$, $E^{(4)}$ contains $bc$ plus perhaps some edges
incident with $a$, and every other edge is in $E^{(1)}$;

$(e)$ there are four vertices $a,b,c,d$ such that $\{ab\}\subseteq
E^{(2)}\subseteq \{ab, cd\}$, $E^{(3)}=\{ac, bd\}$, $E^{(4)}=\{ad,
bc\}$ and every other edge is in $E^{(1)}$;

$(f)$ $n=5$, $V(K_n) = \{a, b, c, d, e\}$, $E^{(1)} =\{ad, ae, bc\}$,
$E^{(2)}=\{bd, be, ac\}$, $E^{(3)}=\{cd, ce, ab\}$ and
$E^{(4)}=\{de\}$.
\end{theorem}

Li et al. \cite{LiWangLiu} got some exact values and bounds of
$\operatorname{gr}_k(P_5:K_t)$, and investigated the edge-colorings
of complete graphs and complete bipartite graphs without rainbow
$4$-path and $5$-path. Fujita and Magnant \cite{FujitaMagnant2012}
got the structural theorem for $G=S_3^{+}$ like Theorem
\ref{Thm:G-Part}. Li and Wang \cite{LiWangDiscrete} studied the
monochromatic stars in rainbow $K_3$-free and $S_3^+$-free
colorings.

We now give the definitions of two graph classes.
\begin{itemize}
\item[] The graph $S_t^{r}$ is obtained from a star of order $t$ by
adding an extra $r$ independent edges between the leaves of the so
that there are $r$ triangles and $t - 2r - 1$ pendent edges in
$S_t^{r}$. For $r=0$ we obtain $S_t^{r} = K_{1,t-1}$, which are
called \emph{stars}. For $r = \frac{t-1}{2}$, if $t$ is odd we
obtain $S_t^{r} = F_{\frac{t-1}{2}}$, which are called \emph{fans}.

\item[] A \emph{pineapple} $PA_{t,\omega}$ is a graph obtained from the
complete graph $K_\omega$ by attaching $t-\omega$ pendent vertices
to the same vertices of $K_\omega$, we suppose that $t\geq
\omega+1$.
\end{itemize}

In Section $2$, we get some exact values or bounds for
$\operatorname{gr}_k(P_5:H)$, where $H$ is a general graph. In
Section $3$, we obtain some results when $H=S_t^{r}$, where $t\geq
2r+2$ and $r\geq 1$. We also get some results in Section $4$ when
$H$ is a pineapple.

\section{General results}

In this section, we assume that $H$ is a graph of order $t$.

\begin{theorem}\label{th2-1}
For two integers $k,t$ with $k\geq 7$ and $k\geq t+1$, we have
$$
\operatorname{gr}_k(P_5:H)=\left\lceil\frac{1+\sqrt{1+8k}}{2}\right\rceil.
$$
\end{theorem}

\begin{proof}
Let $N_k$ be an integer with
$N_k=\lceil\frac{1+\sqrt{1+8k}}{2}\rceil$. For the lower bound, if
there is a $k$-edge-coloring $\chi$ of a complete graph $K_{N_k-1}$,
then $k\leq \tbinom{N_k-1}{2}$, contradicting with
$N_k=\lceil\frac{1+\sqrt{1+8k}}{2}\rceil$. It follows that
$\operatorname{gr}_k(P_5:H)\geq N_k$.

It suffices to show that $\operatorname{gr}_k(P_5:H)\leq N_k$. Let
$\chi$ be any $k$-edge-coloring of $K_n \ (n\geq N_k)$ containing no
rainbow copy of $P_5$. From Theorem \ref{th-path-Structure2}, $(b)$
or $(c)$ is true. If $(b)$ is true, let $V^{(2)},V^{(3)},\ldots,V^{(k)}$ be a
partition of $V(K_n)$ such that there are only edges of color $1$ or
$i$ within $V^{(i)}$ for $2\leq i\leq k$, and hence there are only edges
of color $1$ among the parts. Choose one vertex of $V^{(i)}$, say $v_i$.
Then the subgraph induced by $\{v_2,v_3,\ldots,v_k\}$ is a complete
graph $K_{k-1}$. For $k\geq t+1$, there exists a monochromatic copy
of $K_{t}$ colored by $1$, and hence there is a monochromatic copy
of $H$ colored by $1$. If $(c)$ is true, then there is a vertex $v$
such that $K_{n}-v$ is monochromatic and $n\geq k$. For $k\geq t+1$
and $n\geq t+1$, there is a monochromatic copy of $K_t$, and so we
can find a monochromatic copy of $H$.
\end{proof}

\begin{theorem}\label{th2-2-1}
Let $k,t$ be two integers with $k=5,6$, $k\geq t+1$ and $t\geq 3$.
Then $\operatorname{gr}_k(P_5:H)=5$.
\end{theorem}

\begin{proof}
For the lower bound, we first suppose $k=5$. Let $G_1$ denote a
colored complete graph $K_4$ with $V(K_4)=\{v_i\,|\,1\leq i\leq 4\}$
under the $5$-edge-coloring $\chi$ such that
$\chi(v_1v_2)=\chi(v_3v_4)=1$, $\chi(v_1v_3)=2$, $\chi(v_1v_4)=3$,
$\chi(v_2v_3)=4$ and $\chi(v_2v_4)=5$. Next, we suppose $k=6$. Let
$G_2$ denote a colored complete graph $K_4$ with
$V(K_4)=\{v_i\,|\,1\leq i\leq 4\}$ under the $6$-edge-coloring
$\chi$ such that $\chi(v_1v_2)=1$, $\chi(v_3v_4)=2$,
$\chi(v_1v_3)=3$, $\chi(v_1v_4)=4$, $\chi(v_2v_3)=5$ and
$\chi(v_2v_4)=6$. Since both $G_1$ and $G_2$ contain neither a
rainbow copy of $P_5$ nor a monochromatic copy of $H$, it follows
that $\operatorname{gr}_k(P_5:H)\geq 5$ for $k=5,6$.

It suffices to show that $\operatorname{gr}_k(P_5:H)\leq 5$. Suppose
$G$ is any $k \ (k=5,6)$-edge-coloring of $K_n \ (n\geq 5)$ which
contains no rainbow copy of $P_5$. From Theorem
\ref{th-path-Structure2}, $(b)$ or $(c)$ is true, and the proof is
similar to Theorem \ref{th2-1}.
\end{proof}

For $k=t$, we first show that $H$ is not a complete graph.
\begin{theorem}\label{th2-2}
Let $k,t$ be two integers with $k\geq 5$ and $k=t$. Then
$$
\operatorname{gr}_k(P_5:H)=t+1.
$$
\end{theorem}

\begin{proof}
For the lower bound, from Theorem \ref{th2-1}, let $G_3$ be a
complete graph obtained from a $K_{t-1}$ with vertex set
$\{u_1,u_2,\ldots, u_{t-1}\}$ colored with $1$ by adding a new
vertex $u$ and edges $u_iu(1\leq i\leq t-1)$ colored by $i+1$. Note
that there is neither a rainbow copy of $P_5$ nor a monochromatic
copy of $H$. Thus $\operatorname{gr}_k(P_5:H)\geq t+1$.

It suffices to show that $\operatorname{gr}_k(P_5:H)\leq t+1$. Let
$\chi$ be any $k$-edge-coloring of $K_n \ (n\geq t+1)$ containing no
rainbow copy of $P_5$. From Theorem \ref{th-path-Structure2}, $(b)$
or $(c)$ is true. If $(b)$ is true, choose one vertex of $V^{(i)}(2\leq i\leq k)$, say $v_i$. Then $\{v_2,v_3,\ldots,v_k\}$ induced a complete graph
$K_{k-1}$. For $k\geq t$, there exists a monochromatic copy of
$K_{t-1}$ with color $1$. As $|V^{(i)}|\geq 2$, we can choose another
vertex of $V^{(2)}$, say $u$. Note that $\{u,v_2,v_3,\ldots,v_k\}$
induced a monochromatic $K_t-e$, $e=uv_2$. Then there is a copy of
$H$ with color $1$. If $(c)$ is true, there is a vertex $v$ such
that $K_n-v$ is monochromatic. For $n\geq t+1$, there is a
monochromatic copy of $K_t$, we can find a monochromatic copy of
$H$.
\end{proof}

Next we obtain the result on $H$ is not a complete graph by the
following lemma.

\begin{lemma}\label{lem2-1}
For $4\leq k\leq a$, if $H_t$ is a graph of order $t$ and $K_a \
(a\geq 3)$ be the maximal clique of $H_t$, then
$\operatorname{gr}_k(P_5:H_t)\geq (a-1)(t-1)+1$.
\end{lemma}
\begin{proof}
Let $G_4$ be a complete graph with $V(G_4)=U_2\cup U_3\cup\cdots\cup
U_a$ such that the graph induced by $U_i \ (2\leq i\leq a)$ is a
complete graph $K_{t-1}$, the set $\{U_2,U_3,\ldots,U_a\}=X_{2}\cup
\cdots \cup X_{k}$ with each $K_{t-1}$ of $X_{j} \ (2\leq j\leq k)$
colored by $j$ and $|X_{j}|=\lfloor\frac{a-1}{k-1}\rfloor$ or
$\lfloor\frac{a-1}{k-1}\rfloor+1 \ (2\leq j\leq k)$,
$|X_2|+|X_3|+\cdots+|X_k|=a-1$, all edges between $U_i$ and $U_s \
(i\neq s)$ are colored by $1$. Thus $|V(G_4)|=(a-1)(t-1)$. It is
clear that $G_4$ contains neither a rainbow copy of $P_5$ nor a
monochromatic copy of $H_t$, and so
$\operatorname{gr}_k(P_5:H_t)\geq (a-1)(t-1)+1$.
\end{proof}

Next, we suppose that $H$ is a complete graph $K_t$.

\begin{theorem}\label{th2-4}
For two integers $k,t$ with $k\geq 5$ and $k=t$,
$\operatorname{gr}_k(P_5:K_t)=(t-1)^2+1$.
\end{theorem}

\begin{proof}
From Lemma \ref{lem2-1}, we have
$\operatorname{gr}_k(P_5:H)\geq(t-1)^2+1$. It suffices to show that
$\operatorname{gr}_k(P_5:H)\leq (t-1)^2+1$. Suppose that $\chi$ is
any $k$-edge-coloring of $K_n \ (n\geq (t-1)^2+1)$ containing no
rainbow copy of $P_5$. From Theorem \ref{th-path-Structure2}, $(b)$
or $(c)$ is true. If $(b)$ is true, we can choose one vertex of $V^{(i)}
\ (2\leq i\leq k)$, say $v_i$, then $\{v_2,v_3,\ldots,v_k\}$ induces
a monochromatic copy of $K_{t-1}$. If there is a vertex of $V^{(i)} \
(2\leq i\leq k)$, say $V^{(2)}$ and $u_2\in V^{(2)}$, such that
$\chi(v_2u_2)=1$, then the graph induced by
$\{u_2,v_2,v_3,\ldots,v_k\}$ is a monochromatic copy of $K_t$
colored by $1$. If the graph induced by $V^{(i)} \ (2\leq i\leq k)$
contains no edges colored by $1$, as $n\geq (t-1)^2+1$, there is
$|V^{(i)}|\geq t \ (2\leq i\leq k)$, say $|V^{(2)}|\geq t$. Thus the graph
induced by $V^{(2)}$ is a monochromatic copy of $K_t$ colored by $2$.
\end{proof}

\begin{remark}\label{rem2-4}
For two integers $k,t$ with $k\geq 5$ and $k=t$, if $H$ is not a
complete graph and $|V(H)|=t$, then
$\operatorname{gr}_k(P_5:K_t)-\operatorname{gr}_k(P_5:H)=(t-1)^2+1-(t+1)=(t-1)(t-2)-1$
can be arbitrarily large.
\end{remark}

From Theorems \ref{th2-1}, \ref{th2-2-1}, \ref{th2-2} and \ref{th2-4}, we obtain
the following corollary.
\begin{corollary}\label{coro2-4}
For integers $k\geq 5$ and $k\geq t$,
$$\operatorname{gr}_k(P_5:H)=
\begin{cases}
\max{\{\lceil\frac{1+\sqrt{1+8k}}{2}\rceil, 5\}}, & k\geq t+1;\\
t+1, & k=\text{t and H is not a complete graph};\\
(t-1)^2+1, & k=\text{t and H is a complete graph}.
\end{cases}
$$
\end{corollary}

The following theorem shows the result on the graph $H$ obtained
from a complete graph $K_t$ by deleting a maximally matching $M$.

\begin{theorem}\label{th2-5}
For two integers $k,t$ with $\lceil\frac{t+2}{2}\rceil\leq k \leq
t-1$ and $k\geq 5$, if $H$ is a graph obtained from a complete graph
$K_t$ by deleting a maximally matching $M$, then
$$
\operatorname{gr}_k(P_5:H)=\max\left\{\left\lceil\frac{1+\sqrt{1+8k}}{2}\right\rceil,t+1\right\}.
$$
\end{theorem}
\begin{proof}
From Theorem \ref{th2-2}, we have $\operatorname{gr}_k(P_5:H)\geq
\lceil\frac{1+\sqrt{1+8k}}{2}\rceil$. Let $G_5$ be a complete graph
obtained from a $K_{t-1}$ with vertex set
$\{w_1,w_2,\ldots,w_{t-1}\}$ colored by $1$ by adding a new vertex
$w$ and edge set $\{w_iw\,|\,1\leq i\leq t-1\}=W_1\cup W_2\cup
\cdots \cup W_k$ with $|W_i|=\lfloor\frac{t-1}{k-1}\rfloor \text{or}
\lfloor\frac{t-1}{k-1}\rfloor+1$, $|W_1|+\cdots +|W_k|=t-1$.
Clearly, $G_5$ contains neither a rainbow copy of $P_5$ nor a
monochromatic copy of $H$, and hence $\operatorname{gr}_k(P_5:H)\geq
t+1$. So $\operatorname{gr}_k(P_5:H)\geq
\max{\{\lceil\frac{1+\sqrt{1+8k}}{2}\rceil,t+1\}}$.

It suffices to show that $\operatorname{gr}_k(P_5:H)\leq
\max{\{\lceil\frac{1+\sqrt{1+8k}}{2}\rceil,t+1\}}$. Let $N$ be an
integer with $N=\max{\{\lceil\frac{1+\sqrt{1+8k}}{2}\rceil,t+1\}}$.
Suppose that $\chi$  is any $k$-edge-coloring of $K_n \ (n\geq N)$
containing no rainbow copy of $P_5$. From Theorem
\ref{th-path-Structure2}, $(b)$ or $(c)$ is true. If $(b)$ is true,
then for $|V^{(i)}|\geq 2 \ (2\leq i\leq k)$, we choose two vertices
from $V^{(i)} \ (2\leq i\leq k)$, say $u_i,v_i$. Then
$\{u_2,v_2,u_3,v_3,\ldots,u_k,v_k\}$ induced a monochromatic copy of
graph $K_{2k-2}\setminus M$ with color $1$, where $M$ is a maximally
matching of $K_{2k-2}$. Since $2k-2\geq t$, it follows that there is
a copy of graph $H$ with color $1$. If $(c)$ is true, then there is
a vertex $v$ such that $K_n-v$ is monochromatic. For $n\geq t+1$,
there is a monochromatic $K_t$, we can find a monochromatic copy of
$H$.
\end{proof}

The lower and upper bounds of $\operatorname{gr}_k(P_5:H)$ on
$\Delta(H)$ is shown in the following theorem.

\begin{theorem}\label{th2-6}
Let $k,t,p,q$ be four positive integers with $5\leq k\leq t-1$. If
$\Delta(H)-1=p(k-2)+q,q\in\{0,1,\ldots,k-3\}$ and
$\operatorname{R}_2(H)\geq t+1$, then
$$
\max{\{\Delta(H)+p,t+1\}}\leq \operatorname{gr}_k(P_5:H)\leq
\operatorname{R}_2(H).
$$
\end{theorem}

\begin{proof}
We first show that $\operatorname{gr}_k(P_5:H)\geq
\max{\{\Delta(H)+p,t+1\}}$. Let $G_6$ be a complete graph with
$V(G_6)=U_2\cup U_3\cup\cdots\cup U_k$ such that the graph induced
by $U_i \ (2\leq i\leq k)$ is a monochromatic graph with color $i$,
and all edges between $U_i$ and $U_j \ (i\neq j)$ are colored by
$1$, $|U_i|=p+1 \ (2\leq i\leq q+1)$, $|U_i|=p \ (q+2\leq i\leq k)$.
Thus $|V(G_6)|=(k-1)p+q$. Choose any $k-2$ $U_i$'s from
$\{U_2,U_3,\ldots,U_k\}$, say $U_2,\ldots,U_{k-1}$. Then
$|U_2|+\cdots+|U_{k-1}|\leq \Delta(H)-1$. For any vertex of $U_i$,
say $u_i \ (2\leq i\leq k)$, the degree of $u_i$ in $G_6$ is at most
$\Delta(H)-1$. Note that both $G_5$ and $G_6$ have neither a rainbow
copy of $P_5$ nor a monochromatic copy of $H$. So
$\operatorname{gr}_k(P_5:H)\geq \max{\{(k-1)p+q+1,t+1\}}$.

It suffices to show that $\operatorname{gr}_k(P_5:H)\leq
\operatorname{R}_2(H)$. Let $\chi$ be any $k$-edge-coloring of
$K_{n} \ (n\geq \operatorname{R}(H))$ containing no rainbow copy of
$P_5$, From Theorem \ref{th-path-Structure2}, $(b)$ or $(c)$ is
true. We first consider $(b)$ is true. Let $V^{(2)},V^{(3)},\ldots,V^{(k)}$ be a partition of $V(K_n)$ such that there are only edges of color $1$ or
$i$ within $V^{(i)}$ for $2\leq i\leq k$, and there are only edges of
color $1$ between the parts. Now we recolor the edges of $K_n$ to
make a $2$-edge coloring of $K_n$ such that all edges with color $i
\ (3\leq i\leq k)$ of $V^{(i)} \ (3\leq i\leq k)$ are changed to color
$2$. Let $F$ denote the resulting graph. Since $|V(F)|=n\geq
\operatorname{R}_2(H)$, it follows that $K_n$ must contain a
monochromatic copy of $H$. If $(c)$ is true, as
$\operatorname{R}_2(H)\geq t+1$, there is a vertex $v$ such that
$K_{\operatorname{R}_2(H)}-v$ is monochromatic, say color $1$, and
so there is a monochromatic copy of $H$.
\end{proof}

\section{Results for the rainbow $4$-path and monochromatic $S_t^r$}

From the result on general results in Section $2$, we investigate
the case $3\leq k\leq t-1$ for the graph $S_t^{r}$. First we
consider $5\leq k\leq t-1$.

\begin{theorem}\label{th3-1}
Let $k,r,t$ be three integers with $5\leq k\leq t-1$ and $1\leq
r\leq k-2$. Then
$$
\operatorname{gr}_k(P_5:S_t^r)=\max{\{t+p-1,t+1\}},
$$
where $t-2=p(k-2)+q$, $q\in\{0,1,\ldots,k-3\}$.
\end{theorem}

\begin{proof}
From Theorem \ref{th2-6}, we have
$\operatorname{gr}_k(P_5:S_t^r)\geq \max{\{t+p-1,t+1\}}$. It
suffices to show that $\operatorname{gr}_k(P_5:S_t^r)\leq
\max{\{t+p-1,t+1\}}$. Let $N=\max{\{t+p-1,t+1\}}$. Suppose $G$ is
any $k$-edge-coloring of $K_n \ (n\geq N)$ which contains no rainbow
copy of $P_5$. From Theorem \ref{th-path-Structure2}, $(b)$ or $(c)$
is true. We first consider $(b)$ is true. As $N\geq t+p-1$, there
must be $k-2$ $V^{(i)}$'s of $\{V^{(2)},V^{(3)},\ldots,V^{(k)}\}$, say
$V^{(2)},V^{(3)},\ldots,V^{(k-1)}$, such that $|V^{(2)}|+\cdots+|V^{(k-1)}|\geq
t-1$. Choose one vertex $u$ of $V^{(k)}$, and $\ell_i \ (\ell_i\geq 2,
\sum_{2\leq i\leq k-1}\ell_i=t-1)$ vertices of $V^{(i)} \ (2\leq i\leq
t-1)$, and $|V^{(i)}|\geq 2 \ (2\leq i\leq k)$, $\Delta(S_t^r)=t-1$, the
graph induced by these vertices contains a monochromatic copy of
$S_t^r$. If $(c)$ is true, as $n\geq t+1$, there is a vertex $v$
such that $K_{n}-v$ is monochromatic, say color $1$, and hence there
is monochromatic copy of $S_t^r$.
\end{proof}

Next, we show the result on the case $k=4$ for $S_t^r$.

\begin{theorem}\label{th3-2}
Let $k,t,r$ be three integers with $k=4$, $t\geq 6$ and $r=1,2$.
Then
$$
\operatorname{gr}_4(P_5:S_t^r)=t+p-1,
$$
where $t-2=2p+q$ and $q\in\{0,1\}$.
\end{theorem}

\begin{proof}
For the lower bound, let $F_1$ be a $4$-edge-coloring complete graph
with $V(F_1)=U_2\cup U_3\cup U_4$, and the graph induced by
$U_2,U_3,U_4$ are $K_{p+1}$ if $q=1$ or $K_p$ if $q=0$ colored by
$2$, $K_p$ colored by $3$, $K_p$ colored by $4$ respectively, and
there are only edges of color $1$ between the parts. Then
$|V(F_1)|=3p+q$. Let $F_2$ be a complete graph obtained from a
$K_{t-1}$ with vertex set $\{w_1,w_2,\ldots,w_{t-1}\}$ colored by
$1$ by adding a new vertex $w$ and edge set $\{w_iw,1\leq i\leq
t-1\}=W_1\cup W_2\cup W_3\cup W_4$ with
$|W_i|=\lfloor\frac{t-1}{3}\rfloor \text{or}
\lfloor\frac{t-1}{3}\rfloor+1$, $|W_1|+|W_2|+|W_3|+|W_4|=t-1$. Note
that both $F_1$ and $F_2$ contain neither a rainbow copy of $P_5$
nor a monochromatic copy of $S_t^r$. So
$\operatorname{gr}_k(P_5:S_t^r)\geq t+p-1$.

It suffices to show that $\operatorname{gr}_k(P_5:S_t^r)\leq t+p-1$.
Suppose $G$ is any $4$-edge-coloring of $K_n$ where $n\geq t+p-1$
which contains no rainbow copy of $P_5$. From Theorem
\ref{th-path-Structure2}, $(b)$ or $(c)$ or $(d)$ or $(e)$ is true.
Suppose that $(b)$ is true. Let $V^{(2)},V^{(3)},V^{(4)}$ be a partition
of $V(G)$ such that there are only edges of color $1$ or $i$ within
$V^{(i)}$ for $2\leq i\leq 4$. Then there are only edges with color
$1$ among the parts. Since $n\geq t+p-1$, it follows that there
exist two $V^{(i)}$'s of $\{V^{(2)},V^{(3)},V^{(4)}\}$, say $V^{(2)},V^{(3)}$,
such that $|V^{(2)}|+|V^{(3)}|\geq t-1$, otherwise $n<t$, a
contradiction. Choose one vertex of $V^{(4)}$, say $v$, $\ell_1 \
(\ell_1\geq 2)$ vertices of $V^{(2)}$, say $\{u_1,\ldots,u_{\ell_1}\}$
and $\ell_2 \ (\ell_2\geq 2)$ vertices of $V^{(3)}$, say
$\{w_1,\ldots,w_{\ell_2}\}$, $\ell_1+\ell_2=t-1$. Then the subgraph
induced by $\{v,u_1,\ldots,u_{\ell_1},w_1,\ldots,w_{\ell_2}\}$
contains a monochromatic copy of $S_t^r$. Suppose that $(c)$ is
true. Since $n\geq t+p-1$, $t-2=2p+q$ and $q\in\{0,1\}$, it follows
that $n\geq t+1$, and hence there is a vertex $v$ such that
$K_{n}-v$ is monochromatic, say color $1$, and hence there is
monochromatic copy of $S_t^r$. Suppose that $(d)$ or $(e)$ is true.
Since $t\geq 6$ and $n\geq 7$, it follows that there is a
monochromatic copy of $S_t^r$.
\end{proof}

\begin{lemma}\label{le3-1}
$\operatorname{gr}_4(P_5:S_4^1)=6$.
\end{lemma}
\begin{proof}
Let $F_3$ be a colored complete graph $K_5$ with $V(K_5)=\{v_i|1\leq
i\leq 5\}$ under the $4$-edge-coloring $\chi$ such that
$\chi(v_1v_4)=\chi(v_1v_5)=\chi(v_2v_3)=1$,
$\chi(v_1v_3)=\chi(v_2v_4)=\chi(v_2v_5)=2$,
$\chi(v_1v_2)=\chi(v_3v_4)=\chi(v_3v_5)=3$, and $\chi(v_4v_5)=4$.
Since there is neither a rainbow copy of $P_5$ nor a monochromatic
copy of $S_4^1$ under the coloring $\chi$, it follows that
$\operatorname{gr}_4(P_5:S_4^1)\geq 6$. It suffices to show
$\operatorname{gr}_4(P_5:S_4^1)\leq 6$. Suppose $G$ is any
$4$-edge-coloring of $K_n$ where $n\geq 6$ containing no rainbow
copy of $P_5$. From Theorem \ref{th-path-Structure2}, $(b)$ or $(c)$
or $(d)$ or $(e)$ is true. If $(b)$ is true, then there is a
monochromatic $S_4^1$ with color $1$. If $(c)$ is true, then $K_n-v$
is monochromatic, and hence there is a monochromatic $S_4^1$. If
$(d)$ or $(e)$ is true, there is a monochromatic $S_4^1$ with color
$1$.
\end{proof}

\begin{lemma}\label{le3-2}
$\operatorname{gr}_4(P_5:S_5^1)=6$.
\end{lemma}
\begin{proof}
Since $F_3$ contains neither a rainbow copy of $P_5$ nor a
monochromatic copy of $S_5^1$ under the coloring $\chi$, it follows
that $\operatorname{gr}_4(P_5:S_5^1)\geq 6$. It suffices to show
$\operatorname{gr}_4(P_5:S_5^1)\leq 6$. Suppose that $G$ is any
$4$-edge-coloring of $K_n$ where $n\geq 6$ containing no rainbow
copy of $P_5$. From Theorem \ref{th-path-Structure2}, $(b)$ or $(c)$
or $(d)$ or $(e)$ is true. If $(b)$ is true, then there is a
monochromatic $S_5^1$ with color $1$. If $(c)$ is true, then $K_n-v$
is monochromatic, and hence there is a monochromatic $S_5^1$. If
$(d)$ or $(e)$ is true, we can find a monochromatic copy of $S_5^1$
with color $1$.
\end{proof}

The following corollary follows from Theorem \ref{th3-2}, Lemma
\ref{le3-1} and Lemma \ref{le3-2}.
\begin{corollary}\label{co3-1}
Let $k,t,r$ be three integers with $k=4$, $t\geq 6$ and $r=1,2$.
Then
$$\operatorname{gr}_4(P_5:S_t^r)=
\begin{cases}
6, & t=4,5;\\
t+p-1, & t\geq 6.
\end{cases}
$$
\end{corollary}

\begin{theorem}\label{th3-4}
If $r\geq 3$ and $t$ is odd, then
$$\operatorname{gr}_4(P_5:S_t^r)=
\begin{cases}
\frac{3t-5}{2}, & 3\leq r\leq \lfloor\frac{t-1}{4}\rfloor;\\
t+2r-2, & \lceil\frac{t-1}{4}\rceil\leq r\leq \frac{t-3}{2}.
\end{cases}$$
\end{theorem}

\begin{proof}
Suppose $3\leq r\leq \lfloor\frac{t-1}{4}\rfloor$. For the lower
bound, let $F_4$ be a $4$-edge-coloring complete graph with
$V(F_4)=U_2\cup U_3\cup U_4$ and the graph induced by $U_2$ is a
complete graph $K_{\frac{t-1}{2}}$ colored by $2$, and the graph
induced by $U_i \ (i=3,4)$ is a complete graph $K_{\frac{t-3}{2}}$
colored by $i$, and there are only edges of color $1$ between the
parts, and so $|V(G_1)|=\frac{3t-7}{2}$. Since $F_4$ contains no
rainbow copy of $P_5$ and no monochromatic copy of $S_t^r$, it
follows that $\operatorname{gr}_4(P_5:S_t^r)\geq \frac{3t-5}{2}$.

It suffices to show that $\operatorname{gr}_4(P_5:S_t^r)\leq
\frac{3t-5}{2}$. Suppose that $G$ is any $4$-edge-coloring of $K_n$ where
$n\geq \frac{3t-5}{2}$ containing no rainbow copy of $P_5$. From
Theorem \ref{th-path-Structure2}, $(b)$ or $(c)$ or $(d)$ or $(d)$
is true. If $(b)$ is true, let $V^{(2)},V^{(3)},V^{(4)}$ be a partition of
$V(G)$ such that there are only edges of color $1$ or $i$ within
$V^{(i)}$ for $2\leq i\leq 4$, and there are only edges of color $1$
between the parts. Without loss of generality, suppose $|V^{(2)}|\geq
|V^{(3)}|\geq |V^{(4)}|$. If $|V^{(3)}|\leq r-1$, then
$|V^{(4)}|\leq r-1$ and $|V^{(2)}|\geq
\frac{3t-5}{2}-2(r-1)=\frac{3t-4r-1}{2}$.

\begin{claim}\label{claim1}
The subgraph induced by edges with color $1$ of $V^{(2)}$ contains a
maximally matching with at most $r-3$ edges.
\end{claim}
\begin{proof}
Let $M$ be the maximally matching of the subgraph induced by edges
colored by $1$ of $V^{(2)}$. Suppose $|M|\geq r-2$, and $|V^{(3)}|\geq 2$,
$|V^{(4)}|\geq 2$. We can find a monochromatic copy of $S_t^r$ colored
by $1$, a contradiction.
\end{proof}

From Claim \ref{claim1}, there exists no edges colored by $1$ within
$V^{(2)}$ by deleting at most $2r-6$ vertices, and so $|V^{(2)}|-2(r-3)\geq
\lceil\frac{3t-4r-1}{2}\rceil-(2r-6)=\lceil\frac{3t-8r+11}{2}\rceil$.
Since $r\leq \lfloor\frac{t-1}{4}\rfloor$ and
$\lceil\frac{3t-8r+11}{2}\rceil\geq 2r+7$, it follows that there is
a monochromatic copy of $K_{2r+7}$ colored by $2$. For $|V^{(2)}|\geq
\frac{3t-4r-1}{2}\geq t$ and $r\leq \lfloor\frac{t-1}{4}\rfloor$, we
can find a monochromatic copy of $S_t^r$ colored by $2$ within
$V^{(2)}$. Thus we can assume that $|V^{(3)}|\geq r$, and
$|V^{(2)}|\geq|V^{(3)}|\geq r$.

\begin{claim}\label{claim2}
$|V^{(2)}|+|V^{(3)}|\geq t-1$.
\end{claim}
\begin{proof}
Suppose that $|V^{(2)}|+|V^{(3)}|\leq t-2$, then $|V^{(2)}|+|V^{(4)}|\leq t-2$ and $|V^{(3)}|+|V^{(4)}|\leq t-2$, $|V^{(2)}|+|V^{(3)}|+|V^{(4)}|\leq \frac{3t-6}{2}<n$, a
contradiction.
\end{proof}

From Claim \ref{claim2}, choose one vertex of $V^{(4)}$, say $v$,
$\ell_1 \ (\ell_1\geq r)$ vertices of $V^{(2)}$, say
$u_1,\ldots,u_{\ell_1}$ and $\ell_2 \ (\ell_2\geq r)$ vertices of
$V^{(3)}$, say $w_1,\ldots,w_{\ell_2}$, where $\ell_1+\ell_2=t-1$. Then
the graph induced by
$\{v,u_1,\ldots,u_{\ell_1},w_1,\ldots,w_{\ell_2}\}$ contains a
monochromatic copy of $S_t^r$ colored by $1$.

Since $r\geq 3$, $t$ is odd, and $3\leq r\leq
\lfloor\frac{t-1}{4}\rfloor$, it follows that $t\geq 13$, and hence
$n\geq \frac{3t-5}{2}\geq t+4$. For $(b),(c),(d),(e)$, there is a
monochromatic copy of $S_t^r$.

Suppose $\lceil\frac{t-1}{4}\rceil\leq r\leq \frac{t-3}{2}$. For the
lower bound, let $F_5$ be a $4$-edge-coloring complete graph with
$V(F_6)=U_2\cup U_3\cup U_4$ and the graph induced by $U_2$ is a
complete graph $K_{t-1}$ colored by $2$, and the graph induced by
$U_i \ (i=3,4)$ is a complete graph $K_{r-1}$ colored by $i$, and
there are only edges of color $1$ between the parts, and so
$|V(F_5)|=t+2r-3$. Since $F_5$ contains no rainbow copy of $P_5$ and
no monochromatic copy of $S_t^r$, it follows that
$\operatorname{gr}_4(P_5:S_t^r)\geq t+2r-2$.

It suffices to show that $\operatorname{gr}_4(P_5:S_t^r)\leq
t+2r-2$. Suppose $G$ is any $4$-edge-coloring of $K_n$ where $n\geq
t+2r-2$ which contains no rainbow copy of $P_5$. From Theorem
\ref{th-path-Structure2}, $(b)$ or $(c)$ or $(d)$ or $(d)$ is true.
If $(b)$ is true, let $V^{(2)},V^{(3)},V^{(4)}$ be a partition of $V(G)$
such that there are only edges of color $1$ or $i$ within $V^{(i)}$
for $2\leq i\leq 4$, and there are only edges of color $1$ between
the parts. Without loss of generality, suppose $|V^{(2)}|\geq |V^{(3)}|\geq
|V^{(4)}|$. If $2\leq |V^{(3)}|\leq \lfloor\frac{r-4}{2}\rfloor$, then
$|V^{(4)}|\leq \lfloor\frac{r-4}{2}\rfloor$ and $|V^{(2)}|\geq
t+2r-2-2\lfloor\frac{r-4}{2}\rfloor\geq t+r+2$. From Claim
\ref{claim1}, the subgraph induced by edges colored by $1$ of $V_2$
has a maximally matching containing at most $r-3$ edges. Therefore
there exists no edge colored by $1$ within $V^{(2)}$ by deleting at most
$2r-6$ vertices. Since $|V^{(2)}|-(2r-6)\geq t+r+2-(2r-6)=t-r+8\geq
r+11$ for $\lceil\frac{t-1}{4}\rceil\leq r\leq \frac{t-3}{2}$, it
follows that there is a monochromatic copy of $K_{r+11}$. Since
$|V^{(2)}|\geq t+5$, it follows that there is a monochromatic copy of
$S_t^r$ colored by $2$.

If $\lceil\frac{r-4}{2}\rceil\leq |V^{(3)}|\leq r-1$, then $|V^{(2)}|\geq
t+2r-2-2(r-1)=t$. From Claim \ref{claim1}, the subgraph induced by
edges colored by $1$ of $V^{(2)}$ has a maximally matching containing at
most $r-|V^{(3)}|-1$ edges. Therefore, there exists no edge colored by
$1$ within $V^{(2)}$ by deleting at most $2(r-|V^{(3)}|-1)$ vertices. Then
$|V^{(2)}|-2(r-|V^{(3)}|-1)\geq t-2r+2|V^{(3)}|+2\geq r+1$ for
$\lceil\frac{r-4}{2}\rceil\leq |V^{(3)}|\leq r-1$, and so there is a
monochromatic copy of $K_{r+1}$. As $|V^{(2)}|\geq t$, there is a
monochromatic copy of $S_t^r$ colored by $2$.

Thus we can assume that $|V^{(3)}|\geq r$, and $|V^{(2)}|\geq|V^{(3)}|\geq r$.
From Claim \ref{claim2}, $|V^{(2)}|\geq|V^{(3)}|\geq t-1$, choose one vertex
of $V^{(4)}$, say $v$, $\ell_1 \ (\ell_1\geq r)$ vertices of $V^{(2)}$, say
$u_1,\ldots,u_{\ell_1}$ and $\ell_2 \ (\ell_2\geq r)$ vertices of
$V^{(3)}$, say $w_1,\ldots,w_{\ell_2}$. Then the graph induced by
$\{v,u_1,\ldots,u_{\ell_1},w_1,\ldots,w_{\ell_2}\}$ contains a
monochromatic copy of $S_t^r$ with color $1$.

Suppose that $(c)$ is true. Since $r\geq 3$ and $n\geq t+2r-2\geq
t+4$, it follows that there is a vertex $v$ such that $K_n-v$ is a
complete graph colored by $1$, and hence there is a monochromatic
copy of $S_t^r$ colored by $1$.

Suppose that $(d)$ is true. Since $n\geq t+2r-2\geq t+4$, it follows
that $K_n-\{v_1,v_2,v_3\}$ is a complete graph with color $1$, and
hence there is a monochromatic copy of $S_t^r$ with color $1$.

Suppose that $(e)$ is true. Since $n\geq t+2r-2\geq t+4$, it follows
that $K_n-\{v_1,v_2,v_3,v_4\}$ is a complete graph colored by $1$,
and so there is a monochromatic copy of $S_t^r$ colored by $1$.

\end{proof}

\begin{theorem}\label{th3-5}
Let $k,r,t$ be three integers with $k=4$ and $r\geq 3$. If $t$ is
even, then
$$\operatorname{gr}_4(P_5:S_t^r)=
\begin{cases}
\frac{3t-4}{2}, & 3\leq r\leq \lfloor\frac{t}{4}\rfloor;\\
t+2r-2, & \lceil\frac{t}{4}\rceil\leq r\leq \frac{t-2}{2}.
\end{cases}$$
\end{theorem}

\begin{proof}
Suppose $3\leq r\leq \lfloor\frac{t}{4}\rfloor$. For the lower
bound, let $F_6$ be any $4$-edge-coloring complete graph with
$V(F_6)=U_2\cup U_3\cup U_4$ and the graph induced by $U_i \ (2\leq
i\leq 4)$ is a complete graph $K_{\frac{t-2}{2}}$ colored by $i$,
and there are only edges of color $1$ between the parts, thus
$|V(F_6)|=\frac{3t-6}{2}$. Note that $F_6$ contains no rainbow copy
of $P_5$ and no monochromatic copy of $S_t^r$. Thus
$\operatorname{gr}_4(P_5:S_t^r)\geq \frac{3t-4}{2}$. For the upper
bound, the proof is similar to Theorem \ref{th3-4}.

Suppose $\lceil\frac{t}{4}\rceil\leq r\leq \frac{t-2}{2}$. The proof
is similar to Theorem \ref{th3-4}.
\end{proof}

We obtain the result on the case $k=3$ for $S_t^r$ in the following
lemmas.

\begin{lemma}{\upshape\cite{MR1670625, YangPeter, Yang}} \label{le3-3}
$\operatorname{R}_2(K_3,K_5)=14$; $\operatorname{R}_3(S_4^1)=17$;  $\operatorname{R}_3(S_5^1)=21$;
$\operatorname{R}_3(S_6^1)=26$.
\end{lemma}

\begin{lemma}\label{le3-4}
$\max\{5t-4, 2\operatorname{R}_2(S_t^r)-1\}\leq
\operatorname{R}_3(S_t^r)\leq 3\operatorname{R}_2(S_t^r)+6r-6$.
\end{lemma}

\begin{proof}
For the lower bound, let $F_7$ denote a $3$-edge-coloring complete
graph by making five copies of $K_{t-1}$ colored by $1$ and
inserting edges of colors $2$ and $3$ between the copies to form a
unique $2$-edge-coloring $K_5$ which contains no monochromatic
triangle, $|V(F_7)|=5(t-1)$. Let $F'$ be a $2$-edge-coloring
complete graph colored by $1$ and $2$ on $\operatorname{R}_2(S_t^r)-1$
vertices containing no monochromatic copy of $S_t^r$. We construct
$F_8$ by making two copies of $F'$ and inserting all edges between
the copied in color $3$, $|V(F_8)|=2(\operatorname{R}_2(S_t^r)-1)$.
Note that both $F_7$ and $F_8$ contain no monochromatic copy of
$S_t^r$, it follows that $\operatorname{R}_3(S_t^r)\geq \max\{5t-4,
2\operatorname{R}_2(S_t^r)-1\}$.

It suffices to show that $\operatorname{R}_3(S_t^r)\leq
3\operatorname{R}_2(S_t^r)+6r-6$. Suppose $G$ is any $3$-edge-coloring
of $K_n \ (n\geq 3\operatorname{R}_2(S_t^r)+6r-6)$ which is colored by
$1,2,3$. For a any vertex $v\in V(G)$, and $n\geq
3\operatorname{R}_2(S_t^r)+6r-6$, there are at least
$\operatorname{R}_2(S_t^r)+2r-2$ edges incident with $v$ colored by $i
\ (i=1,2,3)$, say $1$. Without loss of generality, the end vertices
of these edges except $v$ are denoted by
$u_1,u_2,\ldots,u_{\operatorname{R}_2(S_t^r)+2r-2}$. Let $G'$ be the
subgraph induced by
$\{u_1,u_2,\ldots,u_{\operatorname{R}_2(S_t^r)+2r-2}\}$, and from
Claim \ref{claim1} of Theorem \ref{th3-4}, the subgraph induced by
edges colored by $1$ of $G'$ contains a maximal matching which has
at most $r-1$ edges, otherwise there is a monochromatic copy of
$S_t^r$ colored by $1$. Thus $G'$ contains no edge colored by $1$ by
deleting at most $2(r-1)$ vertices, and the resulting graph is
denoted by $R$. Since $|V(R)|\geq |V(G')|-2(r-1)\geq
\operatorname{R}(S_t^r)$, it follows that there must be a
monochromatic copy of $S_t^r$ colored by $2$ or $3$, completing the
proof.
\end{proof}

Finally, we show the result on $S_4^1,S_5^1,S_6^1$.

\begin{theorem}\label{th3-6}
For integers $k\geq 3$, we have
$$\operatorname{gr}_k(P_5:S_4^1)=
\begin{cases}
17, & k=3;\\
6, & k=4;\\
5, & k=5,6;\\
\ell, & {\ell-1\choose 2}+1\leq k\leq {\ell\choose 2} \ and \ \ell\geq 5.
\end{cases}$$
\end{theorem}
\begin{proof}
If $k=3$, then it follows from Lemma \ref{le3-3} that
$\operatorname{gr}_3(P_5:S_4^1)=\operatorname{R}_3(S_4^1)=17$. If
$k=4$, then it follows from Lemma \ref{le3-1} that
$\operatorname{gr}_4(P_5:S_4^1)=6$.

Suppose that $k=5$. Let $F_9$ be a colored complete graph $K_4$ with
$V(K_4)=\{u_i|1\leq i\leq 4\}$ under the $5$-edge-coloring $\chi$
such that $\chi(v_2v_3)=\chi(v_2v_4)=1$, $\chi(v_1v_2)=2$,
$\chi(v_1v_3)=3$, $\chi(v_1v_4)=4$ and $\chi(v_3v_4)=5$. Since there
is neither a rainbow copy of $P_5$ nor a monochromatic copy of
$S_4^1$ under the coloring $\chi$, it follows that
$\operatorname{gr}_5(P_5:S_4^1)\geq 5$. It suffices to show
$\operatorname{gr}_5(P_5:S_4^1)\leq 5$. Let $\chi$ be any
$5$-edge-coloring of $K_n \ (n\geq 5)$ containing no rainbow copy of
$P_5$. From Theorem \ref{th-path-Structure2}, $(b)$ or $(c)$ is
true. If $(b)$ is true, then $\chi$ contains at most $3$ colors, a
contradiction. If $(c)$ is true, then there exists a vertex $v$ such
that $K_n-v$ is monochromatic, and hence there is a monochromatic
copy of $S_4^1$.

Suppose $k=6$. Let $F_{10}$ be a colored complete graph $K_4$ with
$V(K_4)=\{u_i\,|\,1\leq i\leq 4\}$ under the $5$-edge-coloring
$\chi$ such that $\chi(v_2v_3)=1$, $\chi(v_1v_2)=2$,
$\chi(v_1v_3)=3$, $\chi(v_1v_4)=4$, $\chi(v_3v_4)=5$ and
$\chi(v_2v_4)=6$. Since there is neither a rainbow copy of $P_5$ nor
a monochromatic copy of $S_4^1$ under the coloring $\chi$, it
follows that $\operatorname{gr}_6(P_5:S_4^1)\geq 5$. It suffices to
show $\operatorname{gr}_6(P_5:S_4^1)\leq 5$. Let $\chi$ be any
$6$-edge-coloring of $K_n(n\geq 5)$ containing no rainbow copy of
$P_5$. From Theorem \ref{th-path-Structure2}, $(b)$ or $(c)$ is
true. If $(b)$ is true, then $\chi$ contains at most $4$ colors, a
contradiction. If $(c)$ is true, then there exists a vertex $v$ such
that $K_n-v$ is monochromatic, and hence there is a monochromatic
copy of $S_4^1$.

Suppose ${\ell-1\choose 2}+1\leq k\leq {\ell\choose 2}$ and $\ell
\geq 5$. For $k\geq {\ell-1\choose 2}+1$, there is no
$k$-edge-coloring $\chi$ of $K_{\ell-1}$, and so
$\operatorname{gr}_k(P_5:S_4^1)\geq \ell$. It suffices to show
$\operatorname{gr}_k(P_5:S_4^1)\leq \ell$. Suppose that there is a
coloring $\chi$ of $K_n(n\geq \ell)$ containing no rainbow copy of
$P_5$. From Theorem \ref{th-path-Structure2}, $(b)$ or $(c)$ is
true. If $(b)$ is true, then $\ell\geq 2(k-1)>2({\ell-1\choose
2}-1)$, and hence $\ell\leq 4$, a contradiction. If $(c)$ is true,
then there is a vertex $v$ such that $K_n-v$ is monochromatic, and
hence there is a monochromatic copy of $S_4^1$.
\end{proof}

\begin{theorem}\label{th3-7}
For integers $k\geq 3$, we have
$$\operatorname{gr}_k(P_5:S_5^1)=
\begin{cases}
21, & k=3;\\
6, & k=4,5;\\
5, & k=6;\\
\left\lceil\frac{1+\sqrt{1+8k}}{2}\right\rceil, & k\geq 7.
\end{cases}$$
\end{theorem}
\begin{proof}
If $k=3$, then it follows from Lemma \ref{le3-3} that
$\operatorname{gr}_3(P_5:S_5^1)=\operatorname{R}_3(S_5^1)=21$. If
$k=4$, then it follows from Lemma \ref{le3-1} that
$\operatorname{gr}_4(P_5:S_5^1)=6$.

Suppose $k=5$. Let $F_{11}$ be a colored complete graph obtained
from a $K_4$ with vertex set $\{u_1,u_2,u_3,u_4\}$ colored by $1$ by
adding a new vertex $v$ and edges $u_iv \ (1\leq i\leq 4)$ colored
by $i+1$. Since there is neither a rainbow copy of $P_5$ nor a
monochromatic copy of $P_5$ under the coloring, it follows that
$\operatorname{gr}_5(P_5:S_5^1)\geq 6$. It suffices to show
$\operatorname{gr}_5(P_5:S_5^1)\leq 6$. Let $\chi$ be any
$5$-edge-coloring of $K_n \ (n\geq 6)$ containing no rainbow copy of
$P_5$. From Theorem \ref{th-path-Structure2}, $(b)$ or $(c)$ is
true. If $(b)$ is true, then $\chi$ contains at most $4$ colors, a
contradiction. If $(c)$ is true, then there exists a vertex $v$ such
that $K_n-v$ is monochromatic, and hence there is a monochromatic
copy of $S_5^1$.

For $k=6$, the result follows from Theorem \ref{th2-2-1}. For $k\geq
7$, the result follows from Theorem \ref{th2-1}.
\end{proof}

\begin{theorem}\label{th3-8}
For integer $k\geq 3$, we have
$$\operatorname{gr}_k(P_5:S_6^1)=
\begin{cases}
26, & k=3;\\
7, & 4\leq k\leq 6;\\
\lceil\frac{1+\sqrt{1+8k}}{2}\rceil, & k\geq 7.
\end{cases}$$
\end{theorem}

\begin{proof}
If $k=3$, then it follows from Lemma \ref{le3-3} that
$\operatorname{gr}_3(P_5:S_6^1)=\operatorname{R}_3(S_6^1)=26$. If
$k=4$, then it follows from Theorem \ref{th3-2} that
$\operatorname{gr}_4(P_5:S_6^1)=7$. If $k=5$, then it follows from
Theorem \ref{th3-1} that $\operatorname{gr}_5(P_5:S_6^1)=7$. If
$k=6$, then it follows from Theorem \ref{th2-2} that
$\operatorname{gr}_6(P_5:S_6^1)=7$. If $k\geq 7$, then it follows
from Theorem \ref{th2-1} that
$\operatorname{gr}_k(P_5:S_6^1)=\lceil\frac{1+\sqrt{1+8k}}{2}\rceil$,
completing the proof.
\end{proof}

From Theorems \ref{th2-1}, \ref{th2-2}, \ref{th3-1}, \ref{th3-2} and
Lemma \ref{le3-4}, we can obtain the result.
\begin{theorem}\label{th3-9}
For integer $k\geq 3, t\geq 6, r=1,2$, we have
$$\operatorname{gr}_k(P_5:S_t^r)=
\begin{cases}
[\max\{5t-4, 2\operatorname{R}(S_t^r)-1\}, 3\operatorname{R}(S_t^r)+6r-6], & k=3;\\
t+p-1, & 4\leq k\leq t-1;\\
t+1, & k=t;\\
\lceil \frac{1+\sqrt{1+8k}}{2}\rceil,  & k\geq t+1.
\end{cases}$$
\end{theorem}

\section{Results for the rainbow $4$-path and monochromatic pineapples}

In this section, we will get some exact values or bounds for
$\operatorname{gr}_k(P_5:H)$ when $H$ is a pineapple.

\begin{theorem}\label{th4-1}
Let $k,t,\omega$ be three integers with $k=\omega$ and $k\geq 4$.
Then
$$
\operatorname{gr}_k(P_5:PA_{t,\omega})=(\omega-1)(t-1)+1.
$$
\end{theorem}
\begin{proof}
From Lemma \ref{lem2-1}, we have
$\operatorname{gr}_k(P_5:PA_{t,\omega})\geq (\omega-1)(t-1)+1$. It
suffices to show that $\operatorname{gr}_k(P_5:PA_{t,\omega})\leq
(\omega-1)(t-1)+1$. Let $G$ be any $k$-edge-coloring of $K_n$ where
$n\geq (\omega-1)(t-1)+1$ which contains no rainbow copy of $P_5$.
From Theorem \ref{th-path-Structure2}, $(b)$ or $(c)$ or $(d)$ or
$(e)$ is true if $k=4$, and $(b)$ or $(c)$ is true if $k\geq 5$.

For $k\geq 4$, suppose that $(b)$ is true. Let $V^{(2)},V^{(3)},\ldots,V^{(k)}$ be a partition of $V(G)$ such that there are only edges of color $1$ or $i$ within
each $V^{(i)}$ for $2\leq i\leq k$, and there are only edges with color
$1$ among the parts. Then there exists some $V^{(i)} \ (2\leq i\leq k)$
with $|V^{(i)}|\geq t$, otherwise $|V^{(2)}|+|V^{(3)}|+\cdots+|V^{(k)}|\leq
(\omega-1)(t-1)+1<n$, a contradiction. Without loss of generality,
let $|V^{(2)}|\geq t$. Suppose that $V^{(2)}$ contains one edge with color
$1$, say $v_2u_2$. Choose one vertex of $V^{(i)} \ (2\leq i\leq k)$, say
$v_i$. Then the subgraph induced by $\{u_2,v_2,v_3,\ldots,v_k\}$ is
a copy of $K_\omega$ with color $1$. Since $|V^{(2)}|\geq t$, it follows
that there is a copy of $PA_{t,4}$ with color $1$. Therefore, $V^{(2)}$
contains no edges with color $1$, and hence the subgraph induced by
$V^{(2)}$ is a monochromatic copy of $K_t$ with color $2$, and so there
is a monochromatic copy of $PA_{t,\omega}$ with color $2$.

Suppose that $(c)$ is true. Since $n\geq (\omega-1)(t-1)+1\geq t+1$ for
$\omega\geq 4$ and $t\geq 5$, it follows that there is a vertex $v$
such that $K_n-v$ is a complete graph colored by $1$, and hence
there is a monochromatic copy of $PA_{t,\omega}$ colored by $1$.

For $k=4$, then $(d)$ or $(e)$ is true. Suppose that $(d)$ is true. Since $n\geq 3t-2\geq t+8$, it follows that $K_n-\{v_1,v_2,v_3\}$ is a complete graph colored by $1$, and so there is a monochromatic copy of $PA_{t,4}$ colored by $1$. Suppose that $(e)$ is true. Since $n\geq 3t-2\geq t+8$, it follows that $K_n-\{v_1,v_2,v_3,v_4\}$ is a complete graph colored by $1$,
and hence there is a monochromatic copy of $PA_{t,4}$ colored by
$1$.

\end{proof}

\begin{theorem}\label{th4-2}
Let $k,t,\omega$ be three integers with $k=4,\omega=5$ and
$t\geq 8$. Then
$$
\operatorname{gr}_4(P_5:PA_{t,\omega})=4t-3.
$$
\end{theorem}

\begin{proof}
From Lemma \ref{lem2-1}, we know that
$\operatorname{gr}_4(P_5:PA_{t,\omega})\geq 4t-3$. It suffices to
show that $\operatorname{gr}_4(P_5:PA_{t,\omega})\leq 4t-3$. Let
$G$ be any $4$-edge-coloring of $K_n$, where $n\geq 4t-3$, which
contains no rainbow copy of $P_5$. From Theorem
\ref{th-path-Structure2}, $(b)$ or $(c)$ or $(d)$ or $(e)$ is true.
Suppose that $(b)$ is true and $|V^{(2)}|\geq |V^{(3)}|\geq |V^{(4)}|$. To avoid a monochromatic copy of $K_5$ colored by $1$, $|V(G)|-\sum^{4}_{i=2}|V^{(i)}|\leq 1$. If $|V(G)|-\sum^{4}_{i=2}|V^{(i)}|=0$, $\sum^{4}_{i=2}|V^{(i)}|\geq 4t-3$, $|V^{(2)}|\geq \lceil\frac{4t-3}{3}\rceil$, otherwise
$|V^{(2)}|+|V^{(3)}|+|V^{(4)}|< 4t-3\leq n$, a contradiction.

\begin{claim}\label{Claim3}
Both $V^{(3)}$ and $V^{(4)}$ contain no edges with color $1$.
\end{claim}
\begin{proof}
Assume, to the contrary, that both $V^{(3)}$ and $V^{(4)}$ contain one edge
with color $1$, say $v_3u_3$ and $v_4u_4$. Choose one vertex of
$V^{(2)}$, say $v_2$. Then the subgraph induced by
$\{v_2,v_3,u_3,v_4,u_4\}$ contains a copy of $K_5$ with color $1$.
Since $|V^{(2)}|\geq \lceil\frac{4t-3}{3}\rceil\geq t+1$ for $\omega=5$
and $t\geq 8$, it follows that there is a copy of $PA_{t,5}$ with
color $1$. If either $V^{(3)}$ or $V^{(4)}$ contains one edge with color
$1$, say $V^{(3)}$ and $v_3u_3$ with color $1$, then $V^{(2)}$ contains no
edges with color $1$, otherwise there is a copy of $PA_{t,5}$ with
color $1$ by the above proof. Thus the subgraph induced by $V^{(2)}$ is
monochromatic copy of complete graph with color $2$, and hence
$|V^{(2)}|\geq \lceil\frac{4t-3}{3}\rceil\geq t+1$, and so there is a
monochromatic copy of $PA_{t,5}$ with color $2$.
\end{proof}

From Claim \ref{Claim3}, both $V^{(3)}$ and $V^{(4)}$ contain no
edges with color $1$. Then $|V^{(3)}|\leq t-1$, $|V^{(4)}|\leq t-1$,
otherwise there is a monochromatic copy of $K_t$ colored by $3$ or
$4$. If $V^{(2)}$ contains three vertices which induce a
monochromatic copy of $K_3$ with color $1$, say $u_2,w_2,x_2$,
choose one vertex $w_i$ of $V^{(i)}(i=3,4)$, then the subgraph
induced by $\{u_2,w_2,x_2,w_3,w_4\}$ is a monochromatic copy of
$K_5$ with color $1$. Choose $t-5$ vertices of $V^{(2)}$, say
$y_1,\ldots,y_{t-5}$, then the subgraph induced by
$\{u_2,w_2,x_2,w_3,w_4,y_1,\ldots,y_{t-5}\}$ contains a
monochromatic copy of $PA_{t,5}$ colored by $1$. It follows that the
subgraph induced by color $1$ within $V_2$ must not be $K_3$. Since
$n\geq 4t-3$, it follows that $|V^{(2)}|\geq 2t-1\geq t+7\geq 15$
for $t\geq 8$. From Lemma \ref{le3-3},
$\operatorname{R}_2(K_3,K_5)=14$, and hence there is a monochromatic
copy of $K_5$ with color $2$ in $V^{(2)}$.

For any vertex $v\in V^{(2)}$, let $Q_i \ (i=1,2)$ be the set of
vertices such that the edges from any vertex of $Q_i$ to $v$ is with
color $i$.

\begin{claim}\label{Claim4}
$|Q_1|\leq t-1$.
\end{claim}
\begin{proof}
Assume, to the contrary, that $|Q_1|\geq t$. If the subgraph induced
by $Q_1$ contains one edge with color $1$, say $w_1w_2$, then
$\{w_1,w_2,v\}$ induces a monochromatic copy of $K_3$ colored by
$1$, a contradiction. That means that $Q_1$ contains no edges
colored by $1$ and there is a monochromatic copy of $K_t$ colored by
$2$. So there is a monochromatic copy of $PA_{t,5}$ colored by $2$.
\end{proof}

From Claim \ref{Claim4}, we have $|Q_2|\geq t-1$, and hence
$V^{(2)}$ contains a monochromatic copy of $PA_{t,5}$ colored by
$2$. If $|V(G)|-\sum^{4}_{i=2}|V^{(i)}|=1$, then $V^{(2)}$ contains
no edge colored by $1$, otherwise there is a monochromatic copy of
$K_5$. Since $\sum^{4}_{i=2}|V^{(i)}|\geq 4t-4$, it follows that
$|V^{(2)}|\geq \lceil\frac{4t-4}{3}\rceil\geq t+1$ for $t\geq 8$,
and hence $V^{(2)}$ contains a monochromatic copy of $PA_{t,5}$
colored by $2$.

Suppose that $(c)$ or $(d)$ or $e$ is true. Since $n\geq 4t-3\geq
t+21$ for $t\geq 8$, it follows that there is a monochromatic copy
of $K_{t+17}$.

\end{proof}

\begin{theorem}\label{th4-3}
Let $k,t,\omega$ be three integers with $k=4$, $\omega=5$ and $t=6$.
Then
$$
\operatorname{gr}_4(P_5:PA_{6,5})=24.
$$
\end{theorem}

\begin{proof}
For the lower bound, let $F_{12}$ be a complete graph with $V(F_{12})=U_2\cup U_3\cup U_4$, the subgraph induced by $U_2$ is $K_{13}$ colored by $1$ and $2$ which contains neither a monochromatic copy of $K_3$ nor a monochromatic copy of $K_5$, the subgraph induced by $U_i(i=3,4)$ is $K_{5}$ colored by $i$, and all edges between $U_i$ and $U_j(i,j\in \{{2,3,4}\},i\neq j)$ are colored by $1$. It is clear that $F_{12}$ contains neither a rainbow copy of $P_5$ nor a monochromatic copy of $PA_{6,5}$, and $\operatorname{gr}_4(P_5:PA_{6,5})\geq 24$.

It suffices to show that $\operatorname{gr}_k(P_5:PA_{6,5})\leq 24$. Let
$G$ be any $4$-edge-coloring of $K_n$, where $n\geq 24$, which
contains no rainbow copy of $P_5$. From Theorem
\ref{th-path-Structure2}, $(b)$ or $(c)$ or $(d)$ or $(e)$ is true. By the proof of Theorem \ref{th4-2}, we know that the upper bound holds.

\end{proof}

\begin{theorem}\label{th4-4}
Let $k,t,\omega$ be three integers with $k=4$, $\omega=5$ and $t=7$.
Then
$$
\operatorname{gr}_4(P_5:PA_{7,5})=26.
$$
\end{theorem}

\begin{proof}
For the lower bound, let $F_{13}$ be a complete graph with
$V(F_{12})=U_2\cup U_3\cup U_4$. Then the subgraph induced by $U_2$
is $K_{13}$ colored by $1$ and $2$ which contains neither a
monochromatic copy of $K_3$ nor a monochromatic copy of $K_5$, and
the subgraph induced by $U_i \ (i=3,4)$ is $K_{6}$ colored by $i$,
and all edges from $U_i$ to $U_j$ ($i,j\in \{{2,3,4}\}, \ i\neq j$)
are colored by $1$. It is clear that $F_{13}$ contains neither a
rainbow copy of $P_5$ nor a monochromatic copy of $PA_{7,5}$, and
hence $\operatorname{gr}_4(P_5:PA_{7,5})\geq 26$.

It suffices to show that $\operatorname{gr}_k(P_5:PA_{7,5})\leq 26$. Let
$G$ be any $4$-edge-coloring of $K_n$, where $n\geq 26$, which
contains no rainbow copy of $P_5$. From Theorem
\ref{th-path-Structure2}, $(b)$ or $(c)$ or $(d)$ or $(e)$ is true. By the proof of Theorem \ref{th4-2}, we know that the upper bound holds.

\end{proof}

\begin{theorem}\label{th4-5}
Let $k,t,\omega$ be three integers with $k=4$ and $\omega\geq 6$.
Then
$$
(\omega-1)(t-1)+1\leq \operatorname{gr}_4(P_5:PA_{t,\omega})\leq
3\operatorname{R}_2(PA_{t,\omega})-2.
$$
\end{theorem}
\begin{proof}
From Lemma \ref{lem2-1}, we can obtain the lower bound. For the
upper bound, let $G$ be any $4$-edge-coloring $K_n$, where $n\geq
3\operatorname{R}_2(PA_{t,\omega})-2$, containing no rainbow copy of
$P_5$. From Theorem \ref{th-path-Structure2}, $(b)$ or $(c)$ or
$(d)$ or $(e)$ is true. Suppose that $(b)$ is true. Let
$V^{(2)},V^{(3)},V^{(4)}$ be a partition of $V(G)$ such that there
are only edges of color $1$ or $i$ within $V^{(i)}$ for $2\leq i\leq
4$, and there are only edges of color $1$ between the parts. There
exists some $V^{(i)} \ (2\leq i\leq 4)$ with $|V^{(i)}|\geq
\lceil\frac{3\operatorname{R}_2(PA_{t,\omega})-2}{3}\rceil$,
otherwise $|V^{(2)}|+|V^{(3)}|+|V^{(4)}|<
3\operatorname{R}_2(PA_{t,\omega})-2\leq n$, a contradiction.
Without loss of generality, let $|V^{(2)}|\geq
\lceil\frac{3\operatorname{R}_2(PA_{t,\omega})-2}{3}\rceil\geq
\operatorname{R}_2(PA_{t,\omega})$. Then there is a monochromatic
copy of $PA_{t,\omega}$ colored by $1$ or $2$. Since
$\operatorname{R}_2(PA_{t,\omega})\geq t$, it follows that
$3\operatorname{R}_2(PA_{t,\omega})-2>t+4$ for $\omega\geq 6$ and
$t\geq 7$. If $(b)$ or $(c)$ or $(e)$ is true, then there is a
monochromatic copy of $PA_{t,\omega}$.
\end{proof}

Sah \cite{Sah} obtained the following result.
\begin{theorem}{\upshape \cite{Sah}}\label{th4-6}
There is an absolute constant $c>0$ such that for $k\geq 3$,
$$
\operatorname{R}_2(k+1)\leq {2k\choose k}e^{-c(\log k)^2}.
$$
\end{theorem}

By the upper bound in Theorem \ref{th4-5}, we can derive the
following result.
\begin{theorem}\label{th4-7}
There is an absolute constant $c>0$ such that for $\omega\geq 4$,
$$
\operatorname{R}_2(PA_{t,\omega})\leq {2\omega-2\choose
\omega-1}e^{-c\log ^2(\omega-1)}+(t-2)(\omega-1).
$$
\end{theorem}
\begin{proof}
Let $n={2\omega-2\choose \omega-1}e^{-c\log
^2(\omega-1)}+(t-2)(\omega-1)$. For any red/blue-edge-coloring of
$K_n$, from Theorem \ref{th4-4}, there is a monochromatic copy of
$K_{\omega}$, say $K_{\omega}^1$. Without loss of generality, assume
that $K_{\omega}^1$ is red. Choose one vertex in $K_{\omega}^1$, say
$v_1$. Let $X_1$ be the set of vertices with red edges from $v_1$ to
$K_n-K_{\omega}^1$. Then $|X_1|\leq t-\omega-1$. By deleting the
vertices of $K_{\omega}^1\cup X_1$, $K_n-K_{\omega}^1-X_1$ contains
a red clique of order $\omega$, say $K_{\omega}^2$. Choose one
vertex in $K_{\omega}^2$, say $v_2$. Let $X_2$ be the set of
vertices with red edges from $v_2$ to $K_n-(K_{\omega}^1\cup
K_{\omega}^2\cup X_1)$. Then $|X_2|\leq t-\omega-1$. By deleting the
vertices of $K_{\omega}^2\cup X_2$, one can see that
$K_n-(K_{\omega}^1\cup K_{\omega}^2\cup X_1\cup X_2)$ contains a red
clique of order $\omega$, say $K_{\omega}^3$.

Continue this process, $K_n- (\bigcup_{i=1}^{\omega-2}K_{\omega}^i)-
\bigcup_{i=1}^{\omega-2}X_i$ contains a red clique of $\omega$, say
$K_{\omega}^{\omega-1}$. Choose one vertex in
$K_{\omega}^{\omega-1}$, say $v_{\omega-1}$. Let $X_{\omega-1}$ be
the set of vertices with red edges from $v_{\omega-1}$ to
$K_n-(\bigcup_{i=1}^{\omega-1}K_{\omega}^i)-
\bigcup_{i=1}^{\omega-2}X_i$. Then $|X_{\omega-1}|\leq t-\omega-1$.
Choose one vertex in $K_n-(\bigcup_{i=1}^{\omega-1}K_{\omega}^i)-
\bigcup_{i=1}^{\omega-1}X_i$, say $v_{\omega}$. Since the number of
red edges from $v_1$ to $K_n-K_{\omega}^1$ is at most $t-\omega-1$,
it follows that the number of blue edges from $v_1$ to
$K_n-K_{\omega}^1$ is at least $n-t$, and hence there is blue copy
of $PA_{t,\omega}$. Note that the subgraph in $K_n$ induced by
$\{v_{1},v_{2},\ldots,v_{\omega}\}$ is blue clique of order
$\omega$. Since the number of red edges from $v_1$ to
$K_n-K_{\omega}^1$ is at most $t-\omega-1$, it follows that the
number of blue edges from $v_1$ to $K_n-K_{\omega}^1$ is at least
$n-t$, and hence there is blue copy of $PA_{t,\omega}$.
\end{proof}

The following corollary is immediate.
\begin{corollary}
Let $k,t,\omega$ be three integers with $k=4$ and $\omega\geq 6$.
Then
$$
(\omega-1)(t-1)+1\leq \operatorname{gr}_4(P_5:PA_{t,\omega})\leq
3{2\omega-2\choose \omega-1}e^{-c\log
^2(\omega-1)}+(t-2)(\omega-1)-2,
$$
where $c>0$ is an absolute constant.
\end{corollary}

From Lemma \ref{lem2-1}, Theorems \ref{th2-6} and \ref{th4-7}, the
following corollary is true.

\begin{corollary}\label{cor4-4}
Let $k,t,\omega$ be three integers with $5\leq k\leq \omega-1$ and
$\omega\geq 6$. Then
$$
(\omega-1)(t-1)+1\leq \operatorname{gr}_k(P_5:PA_{t,\omega})\leq
{2\omega-2\choose \omega-1}e^{-c\log ^2(\omega-1)}+(t-2)(\omega-1).
$$
\end{corollary}

\end{document}